\documentclass{article}
\usepackage[utf8]{inputenc}

\title{Weak semiconvexity estimates for Schr\"odinger potentials and logarithmic Sobolev inequality for Schr\"odinger bridges}

\usepackage{amsmath,amssymb,amsfonts,amsthm,mathrsfs,comment,xcolor,color,bm}
\usepackage{geometry}
\usepackage[shortlabels]{enumitem}
\usepackage{import}
\newcommand{\bbR}{\mathbb{R}}

\newcommand{\bbRD}{\mathbb{R}^d}

\newcommand{\cF}{\mathcal{F}}

\newcommand{\cH}{\mathcal{H}}

\newcommand{\cU}{\mathcal{U}}

\newcommand{\bes}{\begin{equation*}}
\newcommand{\ees}{\end{equation*}}
\newcommand{\beas}{\begin{eqnarray*}}
\newcommand{\eeas}{\end{eqnarray*}}
\newcommand{\bea}{\begin{eqnarray}}
\newcommand{\eea}{\end{eqnarray}}
\newcommand{\be}{\begin{equation}}
\newcommand{\ee}{\end{equation}}
\newcommand{\bei}{\begin{itemize}}
\newcommand{\eei}{\end{itemize}}
\newcommand{\bec}{\begin{cases}}
\newcommand{\eec}{\end{cases}}
\newcommand{\ben}{\begin{enumerate}}
\newcommand{\een}{\end{enumerate}}

\newcommand{\bbP}{\mathbb{P}}
\newcommand{\bbE}{\mathbb{E}}

\newcommand{\bbQ}{\mathbb{Q}}

\newcommand{\bbl}{\begin{block}}
\newcommand{\ebl}{\end{block}}

\newcommand{\De}{\mathrm{d}}

\newcommand{\rmI}{\mathrm{I}}

\newcommand{\rme}{\mathrm{e}}
\newtheorem{prop}{Proposition}[section]
\newtheorem{theorem}{Theorem}[section]
\newtheorem{lemma}{Lemma}[section]

\newtheorem{remark}{Remark}[section]

\newtheorem{assumption}[theorem]{Assumption}

\newtheorem{cor}{Corollary}[section]

\newcommand{\hjb}[3]{U^{{#1},{#2}}_{#3}}
\newcommand{\ent}{\mathrm{Ent}}
\author{Giovanni Conforti \thanks{CMAP, CNRS, Ecole polytechnique, Institut Polytechnique de Paris, 91120 Palaiseau, France
 \emph{E-mail address}: giovanni.conforti@polytechnique.edu. Research supported by the ANR project  ANR-20-CE40-0014.}, }%%
\begin{document}
\maketitle
\tableofcontents
\newpage
\begin{abstract}
    We investigate the quadratic Schr\"odinger bridge problem, a.k.a. Entropic Optimal Transport problem, and obtain weak semiconvexity and semiconcavity bounds on Schr\"odinger potentials under mild assumptions on the marginals that are substantially weaker than log-concavity. We deduce from these estimates that Schr\"odinger bridges satisfy a logarithmic Sobolev inequality on the product space. Our proof strategy is based on a second order analysis of coupling by reflection on the characteristics of the Hamilton-Jacobi-Bellman equation that reveals the existence of new classes of invariant functions for the corresponding flow.
\end{abstract}

\section{Introduction and statement of the main results}
The Schr\"odinger problem \cite{Schr} (SP) is a statistical mechanics problem that consists in finding the most likely evolution of a cloud of independent Brownian particles conditionally to observations. Also known as Entopic Optimal Transport (EOT) problem and formulated with the help of large deviations theory as a constrained entropy minimization problem, it stands nowadays at the cross of several research lines ranging from functional inequalities \cite{conforti2019second,gentil2020dynamical}, statistical machine learning \cite{cuturi2013sinkhorn,peyre2019computational}, control engineering \cite{chen2014relation,chen2021stochastic}, and numerics for PDEs \cite{benamou2015iterative,benamou2021optimal}. Given two probability distributions $\mu,\nu$ on $\bbRD$, the corresponding (quadratic) Schr\"odinger problem is
\be\label{eq:SP}
\inf_{\pi\in\Pi(\mu,\nu)} \cH(\pi|R_{0T}),
\ee
where $\Pi(\mu,\nu)$ represents the set of couplings of $\mu$ and $\nu$ and $\cH(\pi|R_{0T})$ is the relative entropy of a coupling $\pi$ computed against the joint law $R_{0T}$ at times $0$ and $T$ of a Brownian motion with initial law $\mu$.
It is well known that under mild conditions on the marginals, the optimal coupling $\hat{\pi}$, called (static) Schr\"odinger bridge, is unique and admits the representation
\be\label{eq:opt_coup}
\hat{\pi}(\De x\,\De y) = \exp(-\varphi(x)-\psi(y))\exp\Big(-\frac{|x-y|^2}{2T}\Big)\De x\De y
\ee
 where $\varphi,\psi$ are two functions, known as Schr\"odinger potentials \cite{LeoSch} that can be regarded as proxies for the Brenier potentials of optimal transport, that are recovered in the short-time ($T\rightarrow0$) limit \cite{nutz2022entropic,chiarini2022gradient}. In this article we seek for convexity and concavity estimates for Schr\"odinger potentials. Such estimates have been recently established in \cite{chewi2022entropic} and \cite{fathi2019proof} working under a set of assumptions that implies in particular log-concavity of at least one of the two marginals. Such assumption is crucial therein as it allows to profit from classical functional inequalities such as Prékopa-Leindler inequality and Brascamp-Lieb inequality. In particular, the estimates obtained in the above-mentioned works yield alternative proofs of Caffarelli's contraction Theorem \cite{caffarelli2000monotonicity} in the short-time limit. The purpose of this work is twofold: in { the} first place we show at Theorem \ref{thm:semiconvexity_estimate_Schr_pot} that, for any fixed $T>0$ it is possible to leverage the probabilistic interpretation of \eqref{eq:SP} to establish lower and upper bounds on the functions
 \bes 
\langle  \nabla \varphi(x)-\nabla\varphi(y),x-y \rangle \quad \text{and} \quad \langle  \nabla \psi(x)-\nabla\psi(y),x-y \rangle 
 \ees 
 that are valid for all $x,y\in\bbRD$ and do not require strict log-concavity of the marginals to hold, but still allow to recover the results of \cite{chewi2022entropic} as a special case. The second main contribution is to apply these bounds to prove that static Schr\"odinger bridges satisfy the logarithmic Sobolev inequality (LSI for short) at Theorem \ref{thm:LSI}.  In our main results we shall quantify the weak semiconvexity of a potential $U:\bbRD\longrightarrow\bbR$ appealing to the function $\kappa_U$, defined as follows:
\be\label{eq:kappa_U}
\kappa_U:(0,+\infty)\longrightarrow\bbR,\quad \kappa_{U}(r)= \inf\{|x-y|^{-2}\langle\nabla U(x)-\nabla U(y),x-y\rangle :|x-y|=r\}.
\ee
$\kappa_U(r)$ may be regarded as an averaged or integrated convexity lower bound for $U$ for points that are at distance $r$. This function is often encountered in applications of the coupling method to the study of the long time behavior of Fokker-Planck equations \cite{eberle2016reflection,lindvall1986coupling}. Obviously $\kappa_U\geq0$ is equivalent to the convexity of $U$, but working with non-uniform lower bounds on $\kappa_U$ allows to design efficient generalizations of the classical notion of convexity. A commonly encountered sufficient condition on $\kappa_U$ ensuring the exponential trend to equilibrium of the Fokker-Planck equation 
\bes 
\partial_t\mu_t -\frac{1}{2}\Delta\mu_t -\nabla\cdot\big( \nabla U\,\mu_t\big)=0 
\ees
is the following 
\be\label{eq:Eberle_ass}
\kappa_{U}(r)\geq \begin{cases} \alpha, \quad & \mbox{if $r> R$,}\\
\alpha-L', \quad & \mbox{if $r\leq R$,}
\end{cases}
\ee
for some $\alpha>0,L',R\geq0.$ In this work, we refer to assumptions of the form \eqref{eq:Eberle_ass} and variants thereof as to weak convexity assumptions and our main result require an assumption of this kind, namely \eqref{eq:nu_ass} below, that is shown to be no more demanding than \eqref{eq:Eberle_ass} (see Proposition \ref{prop:gen_ass}), and is expressed through a rescaled version of the hyperbolic tangent function. These functions play a special role in this work since, as we show at Theorem \ref{thm:propagation}, they define a weak convexity property that propagates backward along the flow of the Hamilton-Jacobi-Bellman (HJB) equation
 \bes
    \partial_{t}\varphi_t+\frac{1}{2}\Delta \varphi_t -\frac12|\nabla \varphi_t|^2=0.
 \ees

Such invariance property represents the main innovation in our proof strategy: the propagation of the classical notion of convexity along the HJB equation used in \cite{fathi2019proof} as well as the Brascamp-Lieb inequality employed in \cite{chewi2022entropic} are both consequences of the Prékopa-Leindler inequality, see \cite{BrLieb76}. In the framework considered here, such { a} powerful tool becomes ineffective due to the possible lack of log-concavity in both marginals. To overcome this obstacle we develop a probabilistic approach based on a second order analysis of coupling by reflection on the solutions of the SDE
\bes
\De X_t = -\nabla \varphi_t(X_t)\De t +\De B_t,
\ees
also known as characteristics of the HJB equation, revealing the existence of novel classes of weakly convex functions that are invariant for the HJB flow. This property, besides being a key ingredient in the proof of Theorem \ref{thm:semiconvexity_estimate_Schr_pot} is interesting on its own, and can be generalized in several directions. Remarkably, Theorem \ref{thm:semiconvexity_estimate_Schr_pot} can be aplied to show that Schr\"odinger bridges satisfy LSI: this is not a trivial task since static Schr\"odinger bridges are not known to be log-concave probability measures in general, not even in the case when both marginals are strongly log-concave. For this reason, one cannot infer LSI directly from Theorem \ref{thm:semiconvexity_estimate_Schr_pot} and the Bakry-\'Emery criterion. However, reintroducing a dynamical viewpoint and representing Schr\"odinger bridges as Doob $h$-transforms of Brownian motion \cite{Doob1957} reveals all the effectiveness of Theorem \ref{thm:semiconvexity_estimate_Schr_pot} that gives at once gradient estimates and local (or conditional, or heat kernel) logarithmic Sobolev inequalities and gradient estimates for the $h$-transform semigroup. By carefully mixing the local inequalities with the help of gradient estimates, we finally establish at Theorem \ref{thm:LSI} LSI for $\hat{\pi}$, that is our second main contribution. It is worth noticing that in the $T\rightarrow+\infty$ asymptotic regime, our approach to LSI can be related to the techniques recently developed in \cite{mikulincer2022lipschitz} to construct Lipschitz transports between the Gaussian distribution and probability measures that are approximately log-concave in a suitable sense.
Because of the intrinsic probabilistic nature of our proof strategy, our ability to compensate for the lack of log-concavity in the marginals depends on the size of the regularization parameter $T$, and indeed vanishes as $T\rightarrow0$. Thus, our main results do not yield any sensible convexity/concavity estimate on Brenier potentials that improves on Caffarelli's Theorem. On the other hand, the semiconvexity bounds of Theorem \ref{thm:semiconvexity_estimate_Schr_pot} find applications beyond LSI, that we shall address in future works. For example, following classical arguments put forward in \cite{djellout2004transportation}, they can be shown to imply transport-entropy (a.k.a. Talagrand) inequalities on path space for dynamic Schr\"odinger bridges. Moreover, building on the results of \cite{conforti2019second}, they shall imply new semiconvexity estimates for the Fisher information along entropic interpolations. It is also natural to conjecture that these bounds will provide with new stability estimates for Schr\"odinger bridges under marginal perturbations, thus addressing a question that  has recently drawn quite some attention, see \cite{deligiannidis2021quantitative,chiarini2022gradient,eckstein2021quantitative,ghosal2022convergence, bayraktar2022stability} for example. Finally, we point out that Hessian bounds for potentials can play a relevant role in providing theoretical guarantees for learning algorithms that make use of dynamic Schr\"odinger bridges and conditional processes. In this framework, leveraging Doob's $h$-transform theory and time reversal arguments, they directly translate into various kinds of quantitative stability estimates for the diffusion processes used for sampling, see e.g. \cite{de2021diffusion,de2021simulating,shi2022conditional}.  
%Moreover, these estimates could contribute to advance on the problem of obtaining convergence rates for Sinkhorn's algorithm for non-compact marginals and unbounded costs. 

\paragraph{Organization} The document is organized as follows. In remainder of the first section we {state and discuss} our main hypothesis and results. In Section \ref{sec:propagation} we study invariant sets for the HJB flow. Sections \ref{sec:proofs} and \ref{sec:LSI} are devoted to the proof of our two main results, Theorem \ref{thm:semiconvexity_estimate_Schr_pot} and Theorem \ref{thm:LSI}. Technical results and background material are collected in the Appendix section.
\begin{assumption}\label{ass:marginals}
We assume that $\mu,\nu$ admit a positive density against the Lebesgue measure which can be written in the form $\exp(-U^{\mu})$ and $\exp(-U^{\nu})$ respectively. $U^{\mu},U^{\nu}$ are of class $C^2(\bbRD)$.
\begin{enumerate}
\item[$(\mathbf{H1})$]\label{itm:H1} $\mu$ has finite second moment and finite relative entropy against the Lebsegue measure. Moreover, there exists $\beta_{\mu}>0$ such that
\be\label{eq:mu_ass}
\langle v,\nabla^2U^{\mu}(x)v\rangle \leq \beta_{\mu} |v|^2\quad \forall x,v\in\bbRD.
\ee
\end{enumerate}
One of the following holds
\begin{enumerate}
\item[$(\mathbf{H2})$] There exist $\alpha_{\nu},L>0$ such that
\be\label{eq:nu_ass}
\kappa_{U^{\nu}}(r) \geq \alpha_{\nu}- r^{-1}f_{L}(r) \quad \forall r>0,
\ee
where the function $f_L$ is defined for any $L>0$ by:
\bes
f_L:[0,+\infty]\longrightarrow [0,+\infty], \quad {f_{L}(r) = 2\,L^{1/2}\tanh\Big((rL^{1/2})/2\Big)}.
\ees
\item[$(\mathbf{H2'})$] There exist $\alpha_{\nu}>0,R,L'\geq 0$ such that
\bes
\kappa_{U^{\nu}}(r)\geq \begin{cases} \alpha_{\nu}, \quad & \mbox{if $r> R$,}\\
\alpha_{\nu}-L', \quad & \mbox{if $r\leq R$.}
\end{cases}
\ees
In this case, we set 
\be\label{eq:bar_L_L}
L=\inf\{ \bar{L}:R^{-1}f_{\bar{L}}(R)\geq L'\}.
\ee
\end{enumerate}
\end{assumption}
Clearly, imposing \eqref{eq:nu_ass} is less restrictive than asking that $\nu$ is strongly log-concave. 
\begin{remark}
We  show that $(\mathbf{H2'})$ implies $(\mathbf{H2})$ at Proposition \ref{prop:gen_ass}. { However, since $(\mathbf{H2'})$ is more familiar to most readers we prefer to keep a statement Theorem \ref{thm:semiconvexity_estimate_Schr_pot} that makes use of this assumption.}
\end{remark}
\begin{remark}
The requirement that the density of $\nu$ is strictly positive everywhere could be dropped at the price of additional technicalities. For $\mu$, such requirement is a consequence of \eqref{eq:mu_ass}.
\end{remark}

\paragraph{The Schr\"odinger system} Let $(P_t)_{t\geq 0}$ the semigroup generated by a $d$-dimensional standard Brownian motion. For given marginals, $\mu,\nu$ and $T>0$ the Schr\"odinger system is the following system of coupled non-linear equations
\be\label{eq:schr_syst}
\begin{cases}
\varphi(x) =  U^{\mu}(x) + \log P_T \exp(-\psi)(x), \quad x\in\bbRD,\\
\psi(y) =  U^{\nu}(y) +\log P_T \exp(-\varphi)(y), \quad y\in\bbRD.
\end{cases}
\ee
Under Assumption \ref{ass:marginals}, it is known that the Schr\"odinger system admits a solution $(\varphi,\psi)$, and that if $(\bar\varphi,\bar\psi)$ is another solution, then there exists $c\in\bbR$ such that $(\varphi,\psi)=(\bar\varphi+c,\bar\psi-c)$, see \cite[sec. 2]{nutz2022entropic}\cite{LeoSch} and references therein. The potentials $\varphi,\psi$ are known as \emph{Schr\"odinger potentials} or entropic Brenier potentials in the literature.
\paragraph{Weak semiconvexity and semiconcavity bounds for Schr\"odinger potentials}

In the rest of the article, given a scalar function $U$, any pointwise lower bound on $\kappa_U$  implying in particular that 
\bes
\liminf_{r \rightarrow +\infty} \kappa_U(r)>-\infty
\ees
shall be called a weak semiconvexity bound for $U$. Next, in analogy with \eqref{eq:kappa_U}  we introduce for a differentiable $U:\bbRD\longrightarrow\bbR$ the function $\ell_U$ as follows:
\bes
\ell_U:(0,+\infty)\longrightarrow\bbR,\quad \ell_U(r) = \sup\{ |x-y|^{-2}\langle \nabla U(x)-\nabla U(y), x-y \rangle: |x-y|=r\},
\ees
and call a weak semiconcavity bound for $U$ any pointwise upper bound for $\ell_U$ implying in particular that 
\bes
\limsup_{r\rightarrow+\infty} \ell_U(r)<+\infty.
\ees
Our first main result is about weak semiconvexity and weak semiconcavity bounds for Schr\"odinger potentials.
\begin{theorem}\label{thm:semiconvexity_estimate_Schr_pot}
Let Assumption \ref{ass:marginals} hold and $(\varphi,\psi)$ be solutions of the Schr\"odinger system. Then $\varphi,\psi$ are twice differentiable and for all $r>0$ we have
\be\label{semiconvexity_estimate_Schr_pot}
\kappa_{\psi}(r) \geq \alpha_{\psi} - r^{-1}f_{L}(r),
\ee
\be\label{semiconcavity_estimate_Schr_pot}
\ell_{\varphi} (r)\leq \beta_{\mu} -\frac{ \alpha_{\psi}}{(1+T\alpha_{\psi})}+\frac{r^{-1}f_L(r)}{(1+T\alpha_{\psi})^2},
\ee
where $\alpha_{\psi}>\alpha_{\nu}-1/T$ can be taken to be the smallest solution of the fixed point equation
\be\label{eq:fixed_point}
\alpha= \alpha_{\nu}-\frac{1}{T} + \frac{G(\alpha,2)}{2T^2}, \quad \alpha \in (\alpha_{\nu}-1/T,+\infty)
\ee
where for all $\alpha\geq\alpha_{\nu}-1/T$:
\be\label{eq:def_of_G}
\begin{split}
G(\alpha,u) &= \inf\{ s\geq 0: F(\alpha,s)\geq u\}, \quad u>0,\\
F(\alpha,s) &=\beta_{\mu}s +\frac{ s}{T(1+T\alpha)}+\frac{s^{1/2}f_L(s^{1/2})}{(1+T\alpha)^2}, \quad s>0.
\end{split}
\ee
\end{theorem}
{
There seems to be no closed form expression for the solutions of the fixed point equation \eqref{eq:fixed_point}. However, it is possible to obtain explicit non trivial upper and lower bounds.
\begin{cor}\label{cor:rough}
Let $\bar\alpha$ be a fixed point solution of \eqref{eq:fixed_point}. Then we have
\bes
\frac{\alpha_{\nu}}{2}-\frac1T -\frac{L}{2T^2\alpha_{\nu}\beta_{\mu}}+\frac12 \sqrt{\big(\alpha_{\nu} +\frac{L}{T^2\beta_{\mu}\alpha_{\nu}}\big)^2+\frac{4\alpha_{\nu}}{
T^2\beta_{\mu}}} \leq \bar\alpha \leq \frac{\alpha_{\nu}}{2}-\frac1T+\frac12\sqrt{\alpha^2_{\nu}+\frac{4\alpha_{\nu}}{
T^2\beta_{\mu}}}.
\ees
\end{cor}
}
\begin{remark}
It is proven at Lemma \ref{lem:properties_of_G} that $F(\alpha,\cdot)$ is increasing on $(0,+\infty)$ for all $\alpha>-1/T$. $G(\alpha,\cdot)$ is therefore its inverse.
\end{remark}
%\begin{remark}
%We conjecture that $\alpha_{\psi}$ can be taken to be the largest solution of the fixed point equation \eqref{eq:fixed_point}. To prove so, it would suffice to show that Sinkhorn's iterates (see \cite[Sec 6]{nutz2022entropic}) converge to solutions of the Schr\"odinger system
% under Assumption \ref{ass:marginals} for a large set of initial conditions. We could not find such result in the existing literature.
%\end{remark}
\begin{remark}
    It is possible to check that if $\textbf{(H2)}$ holds with $L=0$, Theorem \ref{thm:semiconvexity_estimate_Schr_pot} recovers the conclusion of \cite[Theorem 4]{chewi2022entropic},after a change of variable. To be more precise, the potentials $(\varphi_{\varepsilon},\psi_{\varepsilon})$ considered there are related to the couple $(\varphi,\psi)$ appearing in \eqref{eq:schr_syst}  by choosing $\varepsilon=T$ and setting 
    \bes
    \varphi_{\varepsilon} = \varepsilon\Big(\varphi-U^{\mu}+\frac{|\cdot|^2}{2\varepsilon} \Big), \quad  \psi_{\varepsilon} = \varepsilon\Big(\psi-U^{\nu}+\frac{|\cdot|^2}{2\varepsilon} \Big).
    \ees
\end{remark}
\begin{remark}
The rescaled potential $T\varphi$ converges to the Brenier potential in the small noise limit \cite{nutz2022entropic}. As explained in the introduction, one cannot deduce from Theorem \ref{thm:semiconvexity_estimate_Schr_pot} an improvement over Caffarelli's Theorem \cite{caffarelli2000monotonicity} by letting $T\rightarrow0$ in Theorem \ref{thm:semiconvexity_estimate_Schr_pot}.
\end{remark}
Our second main result is that the static Schr\"odinger bridge $\hat{\pi}$ satisfies LSI with an explicit constant. We recall here that a probability measure $\rho$ on $\bbR^d$ satisfies LSI with constant $C$ if and only if for all positive differentiable function $f$
\bes
\mathrm{Ent}_{\rho}(f) \leq \frac{C}{2}\int \frac{|\nabla f|^2}{f}\De \rho, \quad \text{where} \quad \mathrm{Ent}_{\rho}(f) = \int f\log f\De \rho - \int f \De \rho \, \log \Big( \int f \De \rho\Big).
\ees
\begin{theorem}\label{thm:LSI}
Let Assumption \ref{ass:marginals} hold and assume furthermore that $\mu$ satisfies LSI with constant $C_{\mu}$. Then the static Schr\"odinger bridge $\hat{\pi}$ satisfies LSI with constant 
\bes
\max\left\{{2}\,C_{\mu},{2}\,C_{\mu}C_{0,T}+ \int_0^T C_{t,T}\,\De t  \right\},
\ees
where for all $t\leq T$
\bes
    C_{t,T}:=\exp\Big(-\int_{t}^T\alpha^{\psi}_s\De s\Big), \quad \alpha^{\psi}_t := \frac{\alpha_{\psi}}{1+(T-t)\alpha_{\psi}}-  \frac{L}{(1+(T-t)\alpha_{\psi})^2},
\ees
and $\alpha_{\psi}$ is as in Theorem \ref{thm:semiconvexity_estimate_Schr_pot}.
\end{theorem}

It is well known that LSI has a number of remarkable consequences including, but certainly not limited to, spectral gaps and concentration of measure inequalities for Lipschitz observables.

    \begin{remark}
    { It is worth noticing that if $U^{\nu}$ is the sum of a strongly convex potential and a Lipschitz perturbation with second derivative bounded below, then \eqref{eq:nu_ass} holds. Moreover, the perturbation needs not to be of bounded support, covering many interesting scenarios as double wells or multiple-wells potentials. At the moment of writing, it is not clear whether or not \eqref{eq:nu_ass} implies that $\nu$ is a bounded or log-Lipschitz perturbation of a strongly log-concave probability measure, a situation where the results of \cite{holley1986logarithmic}\cite{aida1994logarithmicbis} already ensure that $\nu$ satisfies a logarithmic Sobolev inequality.}
\end{remark}

\begin{remark}
    By taking $\mu$ to be a Gaussian distribution, we obtain as a corollary of Theorem \ref{thm:LSI} that any probability $\nu$ fulfilling \eqref{eq:nu_ass} satisfies a logarithmic Sobolev inequality, { though the constant we exhibit here is not optimal. Indeed, the LSI constant for $\nu$ is deduced by marginalization from the LSI constant of $\hat\pi$. Obviously, estimating the LSI constant for $\hat\pi$ is a much more difficult task than estimating the LSI constant of its marginal in particular because $\hat\pi$ does not admit an explicit expression. However, looking at limiting cases, Schr\"odinger potentials become explicit, and the LSI constant for $\nu$ can be more precisely estimated by constructing Lipschitz maps between some nice distribution and $\nu$ arguing on the basis of Theorem \ref{thm:propagation}. In particular, setting $T=1$ and choosing $\mu=\delta_0$ allows to recover the setting in which the "Brownian transport map" \cite{mikulincer2021brownian} is constructed. Changing the reference measure into the stationary Ornstein-Uhlenbeck process, choosing $\mu$ to be the standard Gaussian distribution and setting $T=+\infty$ allows to deploy the technique of heat flow maps \cite{mikulincer2022lipschitz}. These limiting scenarios are in some sense orthogonal to the scope of this work, that is to gain a better understanding on potentials when they cannot be computed in closed form. They are nevertheless of clear interest and will be analyzed in detail in forthcoming work. }
\end{remark}
    %%%%%%%%%%%%%%%%%%%%%%%%%%%%%%%%%%%%%%%%%%%%%%%%%%%%%%%%%%%%%%%%%%%%%%%%%%%%%%%%%%%%

\section{Invariant sets of weakly convex 
functions for the HJB flow}\label{sec:propagation}
We introduce the notation
\be\label{eq:HJB_FK}
\hjb{T}{g}{t}(x):=-\log P_{T-t}\exp(-g)(x) = -\log \left( \frac{1}{(2\pi (T-t))^{d/2}}\int \exp\Big(-\frac{|y-x|^2}{2(T-t)}-g(y)\Big) \De y\right).
\ee
With this notation at hand, \eqref{eq:schr_syst} rewrites as follows:
\be\label{eq:schr_syst_HJB}
\begin{cases}
\varphi =  U^{\mu} - \hjb{T}{\psi}{0},\\
\psi =  U^{\nu} -\hjb{T}{\varphi}{0}.
\end{cases}
\ee
It is well known that under mild conditions on $g$, the map $[0,T]\times\bbRD\ni(t,x)\mapsto \hjb{T}{g}{t}(x)$ is a classical solution of the HJB equation
 \be\label{eq:HJB}
  \begin{cases} 
  \partial_{t}\varphi_t(x)+\frac{1}{2}\Delta \varphi_t(x) -\frac12|\nabla \varphi_t|^2(x)=0, \\
  \varphi_T(x)=g(x).
  \end{cases}
 \ee
 In the next theorem, we construct  for any $L>0$ a set of weakly convex functions $\cF_L$ that is shown to be invariant for the HJB flow. In the proof, and in the rest of the paper we shall denote by $[\cdot,\cdot]$ the quadratic covariation of two It\^o processes.
 
\begin{theorem}\label{thm:propagation}
Fix $L>0$ and define 
\bes
\mathcal{F}_L=\{g\in C^{1}(\bbRD): \kappa_{g}(r) \geq - r^{-1}f_L(r) \quad \forall r>0\}.
\ees
Then for all $0\leq t\leq T<+\infty$ we have
\be\label{eq:propagation_5}
g\in\cF_{L}\Rightarrow\hjb{T}{g}{t}\in\cF_{L}.
\ee
\end{theorem}
{ The fact that convexity of the terminal condition in the HJB equation \eqref{eq:HJB} implies convexity of the solution at all times is equivalent to the fact that the heat flow preserves log-concavity and has been known for a long time, see \cite{BrLieb76}. Theorem \ref{thm:propagation} offers a significant generalization of this property, by showing that there exist weaker properties than pointwise convexity that are transferred from the terminal condition to the solutions of the HJB equation. It can be checked that $f_{L}$ solves the ODE
\be\label{eq:ODE_f}
ff'(r)+2f''(r)=0 \quad \forall r>0, \quad f(0)=0, f'(0)=L.
\ee
To verify the above, it suffices to compute
\bes
f_L'(r) = \frac{L}{\cosh^2(rL^{1/2}/2)}, \quad f''_L(r)=-L^{3/2}\frac{\sinh(rL^{1/2}/2)}{\cosh^3(rL^{1/2}/2)}
\ees
Moreover, we recall here some useful properties of $f_L$:
\be\label{eq:basic properties}
f_L(r)>0,\, f'_L(r)>0,\, f''_L(r)<0,\, f_L(r)\geq rf'_L(r) \quad \forall r>0.
\ee
The condition $ff'(r)+2f''(r)\leq 0$ appears naturally in the main coupling argument of Theorem \ref{thm:propagation} and we have defined the functions $f_L$ \textit{ad hoc} in order to saturate this differential inequality.
}
We are now in position to prove Theorem \ref{thm:propagation}. As anticipated above, the proof relies on the analysis of coupling by reflection along the characteristics of the HJB equation. {In doing so, we heavily rely on a connection with stochastic control. More precisely, the HJB characteristic
\begin{equation*}
    \De X_t=-\nabla \hjb{T}{g}{t}(X_t)\De t+\De B_t, \quad X_0=x,
\end{equation*}
is the optimal process for the stochastic control problem
\bes
\begin{split}
\inf_{(u_s)_{s\in[0,T]}}\,\, &\bbE\Big[\int_{0}^T\frac12|u_s|^2\De s +g(X^{u,x}_T) \Big] \\
        &\text{s.t} \quad \De X^{u,x}_s= u_s \De s +\De B_s, \quad X^{u,x}_0=x.
\end{split}
\ees
In particular, the stochastic maximum principle \cite{Peng90} for this control problem grants that the process $(\nabla \hjb{T}{g}{t}(X_t))_{t\in[0,T]}$ is a martingale, and we will use this fact in the proof of Theorem \ref{thm:propagation} giving a self contained proof for the reader's convenience.}
In the recent article \cite[Thm 1.3]{conforti2022coupling} Hessian bounds for HJB equations originating from stochastic control problems are obtained by means of coupling techniques. These are two-sided bounds that require an a priori knowledge of global Lipschitz bounds on solutions of the HJB equation  to hold. The one-sided estimates of Theorem \ref{thm:propagation} do not require any Lipschitz property of solutions and their proof require finer arguments than those used in \cite{conforti2022coupling}. 
\begin{proof}
We first assume w.l.o.g. that $t=0$ and work under the additional assumption that 
\be\label{eq:propagation_4}
g\in C^3(\bbRD), \quad \sup_{x\in\bbRD}|\nabla^2 g|(x)<+\infty.
\ee
Combining the above with $g\in\cF_L$, we can justify differentiation under the integral sign in \eqref{eq:HJB_FK} and establish that
\bes
 [0,T]\times\bbRD\ni(t,x)\mapsto U^{T,g}_t(x)
\ees 
is a classical solution of \eqref{eq:HJB} such that
\bes 
[0,T]\times\bbRD\ni(t,x)\mapsto \nabla U^{T,g}_t(x)
\ees
 is continuously differentiable in $t$ as well as twice continuously differentiable and uniformly Lipschitz in $x$.
Under these regularity assumptions, for given $x,\hat{x}\in\bbRD$, coupling by reflection of two diffusions started at $x$ and $\hat{x}$ respectively and whose drift field is $-\nabla\hjb{T}{g}{t}$ is well defined, see \cite{eberle2016reflection}. That is to say, there exist a stochastic  process $(X_t,\hat{X}_t)_{0\leq t \leq T}$ with $(X_0,\hat{X}_0)=(x,\hat{x})$ and two Brownian motions $(B_t,\hat{B}_t)_{0\leq t\leq T}$ all defined on the same probability space and such that
\bes
\begin{cases}
\De X_t = -\nabla \hjb{T}{g}{t}(X_t)\De t+\De B_t, \quad & \mbox{for $0\leq t\leq T$,}\\
\De \hat{X}_{t} = -\nabla \hjb{T}{g}{t}(\hat{X}_t)\De t+\De \hat{B}_t,\quad & \mbox{for $0\leq t\leq \tau$, $X_t=\hat{X}_t$ for $t>\tau$,} \\
\end{cases}
\ees
where
\bes
\rme_t = r^{-1}_t (X_t-\hat{X}_t), \quad r_t = |X_t-\hat{X}_t|,  \quad \De \hat{B}_t = \De B_t -2\rme_t \langle\rme_t,\De B_t\rangle
\ees
and 
\bes
\tau = \inf\{t \in [0,T]: X_{t}=\hat{X}_t \}\wedge T.
\ees
{In particular, $(\hat{B}_t)_{0\leq t\leq T}$ is a Brownian motion by Lévy's characterization.}
We now define 
\bes
\cU:[0,T]\times\bbRD\times\bbRD\longrightarrow \bbR, \quad \cU_{t}(x,\hat{x}) =\begin{cases} |x-\hat{x}|^{-1}\langle \nabla \hjb{T}{g}{t}(x)- \nabla \hjb{T}{g}{t}(\hat{x}),x-\hat{x}\rangle, \quad &\mbox{if $x\neq\hat{x}$,}\\
0 \quad &\mbox{if $x=\hat{x}$,}
\end{cases}
\ees
and proceed to prove that $(\cU_t(X_t,\hat{X}_t))_{0\leq t\leq T}$ is a supermartingale. To this aim, we first deduce from \eqref{eq:HJB} and It\^o's formula that 
\be\label{eq:propagation_2}
 \De \nabla \hjb{T}{g}{t}(X_t) = \De M_t, \quad \De \nabla \hjb{T}{g}{t}(\hat{X}_t) = \De \hat{M}_t
\ee
where $M_{\cdot},\hat{M}_{\cdot}$ are square integrable martingales. Indeed we find from It\^o's formula
\bes
\begin{split}
\De \nabla \hjb{T}{g}{t}(X_t) &= \Big(\partial_t\nabla \hjb{T}{g}{t}(X_t) -\nabla^2 \hjb{T}{g}{t} \nabla \hjb{T}{g}{t}(X_t) +\frac{1}{2}\Delta \nabla \hjb{T}{g}{t}(X_t) \Big)\De t + \nabla^2 \hjb{T}{g}{t}(X_t)\cdot \De B_t\\
&\stackrel{\eqref{eq:HJB}}{=}\nabla^2 \hjb{T}{g}{t} (X_t)\cdot \De B_t,
\end{split}
\ees
and a completely analogous argument shows that $ \nabla \hjb{T}{g}{t}(\hat{X}_t) $ is a square integrable martingale. We shall also prove separately at Lemma \ref{lem:direction} that
\be\label{eq:propagation_1}
\De \rme_t = - r^{-1}_t\mathrm{proj}_{\rme^{\bot}_t}(\nabla\hjb{T}{g}{t}(X_t)-\nabla\hjb{T}{g}{t}(\hat{X}_t))\De t \quad \forall t <\tau,
\ee
where $\mathrm{proj}_{\rme^{\bot}_t}$ denotes the orthogonal projection on the orthogonal complement of the linear subspace generated by $\rme_t$. Combining together \eqref{eq:propagation_2} and \eqref{eq:propagation_1} we find that $\De\cU_t(X_t,\hat{X}_t)=0$ for $t\geq \tau$, whereas for $t<\tau$
\bes
\begin{split}
\De \cU_t(X_t,\hat{X}_t)&=\langle \nabla \hjb{T}{g}{t}(X_t)- \nabla \hjb{T}{g}{t}(\hat{X}_t),\De \rme_t\rangle \\
&+\langle \rme_t, \De(\nabla \hjb{T}{g}{t}(X_t)- \nabla \hjb{T}{g}{t}(\hat{X}_t))\rangle +\De [(\nabla \hjb{T}{g}{\cdot}(X_{\cdot})- \nabla \hjb{T}{g}{\cdot}(\hat{X}_{\cdot})),\rme_{\cdot}]_t\\
&\stackrel{\eqref{eq:propagation_2}+\eqref{eq:propagation_1}}{=}-r^{-1}_t|\mathrm{proj}_{\rme^{\bot}_t}(\nabla\hjb{T}{g}{t}(X_t)-\nabla\hjb{T}{g}{t}(\hat{X}_t))|^2\De t + \De \tilde{M}_t.
\end{split}
\ees 
proving that $(\cU_t(X_t,\hat{X}_t))_{0\leq t\leq T}$ is a supermartingale. In the above, $ \tilde{M}_{\cdot}$ denotes a square integrable martingale and to obtain the last equality we used that the quadratic variation term vanishes because of \eqref{eq:propagation_1}. Next, arguing exactly as in \cite[Eq. 60]{eberle2016reflection} (see also \eqref{eq:direction_3} below for more details) on the basis of It\^o's formula and invoking \eqref{eq:ODE_f} we get
\bes
\begin{split}
\De f_{L}(r_t) &= [-f'_{L}(r_t)\cU_t(X_t,\hat{X}_t) + 2 f''_{L}(r_t)]\De t + \De N_t\\
                &\stackrel{\eqref{eq:ODE_f}}{=}-f'_{L}(r_t)[\cU_t(X_t,\hat{X}_t)+f_{L}(r_t)]\De t +\De N_t,
\end{split}
\ees
where $N_{\cdot}$ is a square integrable martingale. It then follows that
\be\label{eq:propagation_3}
\De \big( \cU_t(X_t,\hat{X}_t)+f_{L}(r_t)\big) \leq -f'_{L}(r_t)\Big(\cU_{t}(X_t,\hat{X}_t)+f_{L}(r_t)\Big)\De t +\De N_t + \De \tilde{M}_t.
\ee
from which we deduce that the process
\bes
\Gamma_t=\exp\Big(\int_0^tf'_{L}(r_s)\De s\Big)\big( \cU_t(X_t,\hat{X}_t)+f_{L}(r_t)\big)
\ees
is a supermartingale and in particular is decreasing on average. This gives
%\bes
%\begin{split}
%\De \Big(\exp(\int_0^tf'_{L}(r_s)\De s)\big( \cU_t(X_t,\hat{X}_t)+f_{L}(r_t)\big)\Big)=&f'_{L}(r_t)\exp(\int_0^tf'_{L}(r_s)\De s)\big( \cU_t(X_t,\hat{X}_t)+f_{L}(r_t)\big)\De t\\
%&+\exp(\int_0^tf'_{L}(r_s)\De s)\De \big( \cU_t(X_t,\hat{X}_t)+f_{L}(r_t)\big)\\
%&\stackrel{\eqref{eq:propagation_3}}{\leq} \exp(\int_0^tf'_{L}(r_s)\De s)\De\tilde{M}_t.
%\end{split}
%\ees
\bes
\begin{split}
&|x-\hat{x}|^{-1}\langle\nabla\hjb{T}{g}{0}(x)-\nabla\hjb{T}{g}{0}(\hat{x}),x-\hat{x}\rangle+f_L(|x-\hat{x}|)= \bbE[\Gamma_0] \\
&\geq \bbE[\Gamma_T]\geq\bbE\Big[\exp\Big(\int_0^Tf'_{L}(r_s)\De s\Big)\big(|X_T-\hat{X}_T|\kappa_{g}(|X_T-\hat{X}_T|) +f_{L}(|X_T-\hat{X}_T|)\big)\Big]\geq 0,
\end{split}
\ees
where the last inequality follows from $g\in\cF_{L}$. We have thus completed the proof under the additional assumption \eqref{eq:propagation_4}. In order to remove it, consider any $g\in\cF_{L}$. Then there exist $(g^n)\subseteq\cF_{L}$  such that \eqref{eq:propagation_4} holds for any of the $g^n$,  $g^n\rightarrow g$ pointwise and $g^n$ is uniformly bounded below. From this, one can prove that $\nabla \hjb{g^n}{T}{0} \rightarrow \nabla \hjb{g}{T}{0}$ pointwise by differentiating \eqref{eq:HJB_FK} under the integral sign. Using this result in combination with the fact that \eqref{eq:propagation_5} holds for any $g^n$ allows to reach the desired conclusion. 
\end{proof}

\begin{lemma}\label{lem:direction}
Under the same assumptions and notations of Theorem \ref{thm:propagation} we have
\bes
\De \rme_t = - r^{-1}_t\mathrm{proj}_{\rme^{\bot}_t}(\nabla\hjb{T}{g}{t}(X_t)-\nabla\hjb{T}{g}{t}(\hat{X}_t))\De t \quad \forall t<\tau.
\ees
\end{lemma}

\begin{proof}
Recall that if $\theta:\bbRD\rightarrow\bbR$ is the map $z\mapsto|z|$, then we have
\be\label{eq:direction_diff_1}
\nabla \theta (z)=  \frac{z}{|z|}, \quad \nabla^2 \theta(z)  = \frac{\rmI}{|z|} - \frac{z z^{\top}}{|z|^3}, \quad z\neq 0.
\ee
The proof consist of several applications of It\^o's formula. We first observe that for $t<\tau$
\be\label{eq:direction_1}
\De (X_t-\hat{X}_t) = -(\nabla\hjb{T}{g}{t}(X_t)-\nabla\hjb{T}{g}{t}(\hat{X}_t))\De t + 2 \rme_t \De W_t, \quad \text{with} \quad \De W_t = \langle\rme_t,\De B_t\rangle.
\ee
Note that $(W_t)_{0\leq t \leq T}$ is a Brownian motion by {Lévy's characterization}. Thus, invoking \eqref{eq:direction_diff_1} (or refferring directly to \cite[Eq. 60]{eberle2016reflection} we obtain
%\bes
%\begin{split}
%\De r_t&= \langle\rme_t \De (X_t-\hat{X}_t)\rangle + 2r^{-1}_t \mathrm{Tr}\big( \rme^{\top}_t (\rmI - \rme_t\rme^{\top}_t ) \rme_t \big)\De t\\
%&=\langle\rme_t \De (X_t-\hat{X}_t)\rangle\De t\\
%&=-\langle\rme_t,(\nabla\hjb{T}{g}{t}(X_t)-\nabla\hjb{T}{g}{t}(\hat{X}_t))\rangle\De t + 2 \De W_t,
%\end{split}
%\ees
\be\label{eq:direction_3}
\De r_t=-\langle\nabla\hjb{T}{g}{t}(X_t)-\nabla\hjb{T}{g}{t}(\hat{X}_t),\rme_t\rangle\De t + 2 \De W_t,
\ee
whence 
\be\label{eq:direction_2}
\begin{split}
\De r^{-1}_t &= -r^{-2}_t\De r_t +  r^{-3}_t \De[r]_t\\
            &=\Big(r^{-2}_t\langle\nabla\hjb{T}{g}{t}(X_t)-\nabla\hjb{T}{g}{t}(\hat{X}_t), \rme_t\rangle+4 r^{-3}_t \Big)\De t -2r^{-2}_t\De W_t.
\end{split}
\ee
Combining \eqref{eq:direction_1} with \eqref{eq:direction_2} we find that for $t<\tau$
\bes
\begin{split}
    \De \rme_t &= \De \big( r^{-1}_t (X_t-\hat{X}_t))\\
    &= r^{-1}_t \De (X_t-\hat{X}_t)+ (X_t-\hat{X}_t)\De ( r^{-1}_t) + \De [X_{\cdot}-\hat{X}_{\cdot},r^{-1}_{\cdot}]_t\\
    &= -r^{-1}_t(\nabla\hjb{T}{g}{t}(X_t)-\nabla\hjb{T}{g}{t}(\hat{X}_t))\De t + 2  r^{-1}_t\rme_t \De W_t\\ &+\Big(r^{-2}_t\langle\nabla\hjb{T}{g}{t}(X_t)-\nabla\hjb{T}{g}{t}(\hat{X}_t), \rme_t\rangle+4 r^{-3}_t \Big)(X_t-\hat{X}_t)\De t\\
    &-2r^{-2}_t(X_t-\hat{X}_t)\De W_t -4r^{-2}_t\rme_t\De t\\
    &=-r^{-1}_t\Big(\nabla\hjb{T}{g}{t}(X_t)-\nabla\hjb{T}{g}{t}(\hat{X}_t)- \langle\nabla\hjb{T}{g}{t}(X_t)-\nabla\hjb{T}{g}{t}(\hat{X}_t),\rme_t\rangle\rme_t\Big)\De t\\
    &=-(r^{-1}_t) \mathrm{proj}_{\rme_t^{\bot}}(\nabla\hjb{T}{g}{t}(X_t)-\nabla\hjb{T}{g}{t}(\hat{X}_t))\De t.
\end{split}
\ees
\end{proof}
\section{Second order bounds for Schr\"odinger potentials}\label{sec:proofs}
From now on Assumption \ref{ass:marginals} is in force, even if we do not specify it. Moreover, since we show at Proposition \ref{prop:gen_ass} in the appendix that $(\mathbf{H2}')$ implies $(\mathbf{H2})$, we shall always assume that $(\mathbf{H2})$ holds in the sequel. The next two subsections are devoted to establish the key estimates needed in the proof of Theorem \ref{thm:semiconvexity_estimate_Schr_pot}, that is carried out immediately afterwards.
\subsection{Weak semiconvexity of $\psi$ implies weak semiconcavity of $\varphi$}
{
We begin this section with a useful reminder of the definition of $F$, first given at \eqref{eq:def_of_G}.
\bes
F(\alpha,s) =\beta_{\mu}s +\frac{ s}{T(1+T\alpha)}+\frac{s^{1/2}f_L(s^{1/2})}{(1+T\alpha)^2}, \quad s>0.
\ees
}
\begin{lemma}\label{lem:concavity_propagation}
Assume that $\alpha>-1/T$ exists such that
\bes
\kappa_{\psi}(r) \geq \alpha- r^{-1}f_L(r) \quad \forall r>0.
\ees
{Then we have
\be\label{eq:concavity_propagation_1}
\kappa_{U^{T,\psi}_0}(r)  \geq \frac{\alpha}{1+T\alpha}-  \frac{r^{-1}f_L(r)}{(1+T\alpha)^2} \quad \forall r>0.
\ee}
In particular, 
\be\label{eq:concavity_propagation_11}
\ell_{\varphi} (r)\leq \beta_{\mu}-\frac{\alpha}{1+T\alpha}+  \frac{r^{-1}f_L(r)}{(1+T\alpha)^2} =r^{-2}F(\alpha,r^2)-\frac1T \quad \forall r>0.
\ee
\end{lemma}
\begin{proof}
We define 
\bes \hat{\psi}(\cdot)=\psi(\cdot)-\frac{\alpha}{2}|\cdot|^2.\ees
and note that $\hat{\psi}\in\cF_L$ by construction. We claim that
\be\label{eq:space_time_transformation_HJB}
U^{T,\psi}_0(x)= \frac{\alpha}{2(1+T\alpha)}|x|^2+ U^{T/(1+T\alpha),\hat{\psi}}_0((1+T\alpha)^{-1}x)+ C,
\ee
where $C$ is some constant independent of $x$.
Indeed we have
\bes
\begin{split}
U^{T,\psi}_0(x) -\frac{d}{2}\log(2\pi T)&=-\log \int \exp \Big( -\frac{|y-x|^2}{2T}-\frac{\alpha}{2}|y|^2-\hat{\psi}(y)\Big)\De y\\
&=-\log \int \exp \Big(-\frac{\alpha|x|^2}{2(1+T \alpha)} -\frac{1+T \alpha}{2T}|y-(1+T \alpha)^{-1}x|^2-\hat{\psi}(y)\Big)\De y\\
&=\frac{\alpha|x|^2}{2(1+T \alpha)} +U^{T/(1+T \alpha),\hat{\psi}}_0((1+T \alpha)^{-1}x)-\frac{d}{2}\log (2\pi T/(1+T \alpha ))
\end{split}
\ees
{ Since $\hat{\psi}\in\cF_L$, we can invoke Theorem \ref{thm:propagation} in \eqref{eq:space_time_transformation_HJB} to prove \eqref{eq:concavity_propagation_1}. The estimate \eqref{eq:concavity_propagation_11} is immediately deduced from \eqref{eq:concavity_propagation_1}} recalling the relation \eqref{eq:schr_syst_HJB} and using Assumption \ref{ass:marginals}.
\end{proof}

\subsection{Weak semiconcavity of $\varphi$ implies weak semiconvexity of $\psi$}

We begin by recording some useful properties of the functions $F(\cdot,\cdot)$ and $G(\cdot,\cdot)$.
\begin{lemma}\label{lem:properties_of_G}
Let $T,\beta_{\mu}>0,L\geq0$ be given.
\begin{enumerate}[label=(\roman*)]
\item\label{itm:concavity} For any $\alpha>-1/T$ the function
\bes 
s\mapsto F(\alpha,s)
\ees
is concave and increasing $[0,+\infty)$.
%\item The function 
%\bes
%(-1/T,+\infty)\ni\alpha\mapsto((1/2T+\beta_{\mu}/2)\mathrm{id}_{\bbR_{\geq 0}}+h_{\alpha,L,T}/2)^{-1}(1)
%\ees
%is convex and positive.
\item\label{itm:2}$\alpha\mapsto G(\alpha,2)$ is positive and non decreasing over $(-\frac1T,+\infty)$. { Moreover, 
\be\label{eq:properties_of_G_1}
\sup_{\alpha>-1/T} G(\alpha,2) \leq \frac{1}{2\beta_{\mu}}.
\ee}
\item\label{itm:3} The fixed point equation \eqref{eq:fixed_point} admits at least one solution on $(\alpha_{\nu}-1/T,+\infty)$ and $\alpha_{\nu}-1/T$ is not an accumulation point for the set of solutions.
\end{enumerate}
\end{lemma}
\begin{proof}
We begin with the proof of \ref{itm:concavity}. To this aim, we observe that $f_{L}$ is increasing on $[0,+\infty)$ and therefore so is $s\mapsto s^{1/2}f_L(s^{1/2})$. Therefore 
\bes
 \frac{\De}{\De s}F(\alpha,s)  \geq\beta_{\mu} + \frac{1}{T(1+T\alpha)}> 0,
\ees
where we used $\alpha>-1/T$ in the last inequality. To prove concavity, we observe that 
\bes
\frac{\De^2}{\De u^2}\Big(u^{1/2}f_L(u^{1/2})\Big)\Big|_{u=s}=\frac{s^{-1/2}}{4}f^{''}_L(s^{1/2})+\frac{s^{-3/2}}{4}(f^{'}_L(s^{1/2})s^{1/2}-f_L(s^{1/2})) \stackrel{\eqref{eq:basic properties}}{<} 0.
\ees
Thus $s\mapsto s^{1/2}f_L(s^{1/2})$ is concave and so is $F(\alpha,\cdot)$. We now move on to the proof of \ref{itm:2} by first showing that $G(\cdot,2)$ is positive and then showing that it is increasing. If this was not the case then $G(\alpha,2)= 0$ for some $\alpha>-1/T$ and therefore there exists a sequence $(s_n)_{n\geq0}$ such that $s_n\rightarrow0$ and $F(\alpha,s_n)\geq2$. But this is impossible since $\lim_{s\downarrow 0}F(\alpha,s_n)=0.$  Next, we observe that $F(\alpha,s)$ is increasing in $s$ from item $(i)$ and  decreasing in $\alpha$ for $\alpha\in(-1/T,+\infty)$. For this reason, for any $u$ and $\alpha'\geq \alpha$ we have
\bes
\{s: F(\alpha',s)\geq  u\} \subseteq \{s: F(\alpha,s)\geq  u\} 
\ees
and therefore 
\bes
G(\alpha',u) \geq G(\alpha,u).
\ees
{ We complete the proof of \ref{itm:2} by showing that \eqref{eq:properties_of_G_1} holds. To see this, using $f_{L}(r)\geq 0$ we obtain that for any $\alpha>-1/T$ 
\bes
F(\alpha,s)\geq  \beta_{\mu}s \quad \forall s>0. 
\ees
But then we obtain directly from \eqref{eq:def_of_G} that
$G(\alpha,2) \leq 1/(2\beta_{\mu})$ thus proving \eqref{eq:properties_of_G_1}.}
To prove \ref{itm:3}, we introduce 
\bes
h:[\alpha_{\nu}-\frac1T,+\infty)\longrightarrow\bbR, \quad h(\alpha) := \alpha - \Big( \alpha_{\nu}-\frac1T+\frac{G(\alpha,2)}{2T^2}\Big)
\ees
Note that that $h$ is continuous on its domain since $G(\cdot,2)$ is so. Therefore, to reach the conclusion it suffices to show that 
\be\label{eq:properties_of_G_2}
h\Big(\alpha_{\nu}-\frac1T\Big)<0, \quad \lim_{\alpha\rightarrow+\infty}h(\alpha)=+\infty.
\ee
The first inequality is a direct consequence of $G(\alpha_{\nu}-1/T,2)>0$, that we have already proven. { Finally, the fact that $h$ diverges at infinity is a consequence of \eqref{eq:properties_of_G_1}. } 
\end{proof}
We shall now introduce the modified potential $\bar{\psi}$ as follows

\be\label{eq:change_variable}
\bar\psi(y) = T\Big(\psi(y)-U^{\nu}(y)+\frac{|y|^2}{2T} \Big),
\ee
It has been proven at \cite[Lemma 1]{chewi2022entropic} that the Hessian of $\bar\psi$ relates to the covariance matrix of the conditional distributions of the static Schr\"odinger bridge $\hat{\pi}$. That is to say,
\be\label{eq:convexity_propagation_4}
\nabla^2\bar\psi(y)=\frac1T \mathrm{Cov}_{X\sim\hat{\pi}^y}(X)
\ee
where $\hat{\pi}^y$ is (a version of) the conditional distribution of $\hat{\pi}$ that, in view of \eqref{eq:schr_syst} has the following form:
\be\label{eq:cond_distr}
\hat{\pi}^y(\De x)=\frac{\exp(-V^{\hat{\pi}^y}(x))\De x}{\int\exp(-V^{\hat{\pi}^y}(\bar{x}))\De \bar{x}}, \quad V^{\hat{\pi}^y}(x):=\varphi(x)+\frac{|x|^2}{2T}-\frac{xy}{T}.
\ee
We shall give an independent proof of \eqref{eq:convexity_propagation_4} under additional regularity assumptions at Proposition \ref{prop:hessian_cov} in the Appendix for the readers' convenience. A consequence of \eqref{eq:convexity_propagation_4} is that $\bar{\psi}$ is convex and we obtain from \eqref{eq:change_variable} that 
\be\label{eq:crude_bound}
\kappa_{\psi}(r) \geq \alpha_{\nu}-\frac1T-r^{-1}f_L(r) \quad \forall r>0.
\ee
This is a first crude  weak semiconvexity bound on $\psi$ upon which Theorem \ref{thm:semiconvexity_estimate_Schr_pot} improves by means of a recursive argument. We show in the forthcoming Lemma how to deduce weak semiconvexity of $\psi$ from weak semiconcavity of $\varphi$. In the $L=0$ setting, this step is carried out in \cite{chewi2022entropic} invoking the Cramer-Rao inequality, whose application is not justified in the present more general setup.

\begin{lemma}\label{lem:convexity_propagation}
Assume that $\alpha>-1/T$ exists such that
\be\label{eq:convexity_propagation_3}
\ell_{\varphi} (r)\leq-\frac1T + r^{-2}F(\alpha,r^2)\quad \forall r>0.
\ee
Then 
\bes
\kappa_{\psi}(r)\geq \alpha_{\nu}-\frac1T+\frac{G(\alpha,2)}{2T^2}-r^{-1}f_{L}(r) \quad \forall r>0.
\ees
\end{lemma}

\begin{proof}
Recalling the definition of $V^{\hat{\pi}^y}$ given at \eqref{eq:cond_distr} we observe that the standing assumptions imply
{\be\label{eq:convexity_propagation_1}
\ell_{V^{\hat{\pi}^y}}(r)\stackrel{ \eqref{eq:cond_distr} }{\leq} \ell_{\varphi}(r)+ \frac{1}{T} \stackrel{\eqref{eq:convexity_propagation_3}}{=} r^{-2}F(\alpha,r^2) \quad \forall r>0.
\ee}
In view of \eqref{eq:convexity_propagation_4}, we now proceed to bound $\mathrm{Var}_{X\sim\hat{\pi}^y}(X_1)$ from below for a given $y$, where we adopted the notational convention $X=(X_1,\ldots,X_d)$ for the components of random vectors. We first observe { that the variance-bias decomposition formula implies} 
\be\label{eq:convextiy_propagation_2}
\mathrm{Var}_{X\sim\hat{\pi}^y}(X_1) \geq \bbE_{X\sim\hat{\pi}^y}[\mathrm{Var}_{X\sim\hat{\pi}^y}(X_1|X_2,\ldots,X_d)].
\ee
Moreover, { we define for any $z=(z_2,\ldots,z_d)$
\bes
V^{\hat{\pi}^{y,z}}(\cdot):=V^{\hat{\pi}^y}(\cdot, z), \quad  \hat{\pi}^{y,z}(\De x )=\frac{\exp(-V^{\hat{\pi}^{y,z}}(x))\De x}{\int \exp(-V^{\hat{\pi}^{y,z}}(\bar{x}))\De \bar{x}}
\ees
and observe that if $X\sim\hat{\pi}^y$, then the conditional distribution of $X_1$ under given $\{ (X_2,\ldots,X_d)=(z_2,\ldots,z_d)\}$  is precisely $\hat{\pi}^{y,z}$. This gives the formula}
\bes
\mathrm{Var}_{X\sim\hat{\pi}^y}(X_1|X_2=z_2,\ldots,X_d=z_d) =\frac12\int |x-\hat{x}|^2\hat{\pi}^{y,z}(\De x)\hat{\pi}^{y,z}( \De \hat{x})
\ees
{ 
With this notation at hand, and with the help of the following identities, that can be obtained by one-dimensional integration by parts,

\bes
1= \int \partial_{x}V^{\hat{\pi}^{y,z}}(x)\,x\,\, \hat{\pi}^{y,z}(\De x),\quad
0=\int \partial_{x}V^{\hat{\pi}^{y,z}}(x) \hat{\pi}^{y,z}(\De x)
\ees
}
we find that, uniformly in $z\in \bbR^{d-1}$,
\bes
\begin{split}
1&=\frac12\int(\partial_{x}V^{\hat{\pi}^{y,z}}(x) -\partial_{x}V^{\hat{\pi}^{y,z}}(\hat{x}))(x-\hat{x} )\hat{\pi}^{y,z}(\De x)\hat{\pi}^{y,z}( \De \hat{x}) \\
&=\frac12\int\langle \nabla V^{\hat{\pi}^{y}}(x,z) -\nabla V^{\hat{\pi}^{y}}(\hat{x},z),(x,z)-(\hat{x},z) \rangle \hat{\pi}^{y,z}(\De x)\hat{\pi}^{y,z}( \De \hat{x})\\
&\stackrel{\eqref{eq:convexity_propagation_1}}{\leq} \frac12\int F(\alpha,|x-\hat{x}|^2)\hat{\pi}^{y,z}(\De x)\hat{\pi}^{y,z}( \De \hat{x})\\
&\leq \frac12F(\alpha,2\mathrm{Var}_{X\sim\hat{\pi}^y}(X_1|X_2=z_2,\ldots,X_d=z_d))
\end{split}
\ees
where to establish the last inequality  we used { the concavity of $F$ (see Lemma \ref{lem:properties_of_G}\ref{itm:concavity})} and Jensen's inequality. Since $\alpha>-1/T$, invoking again Lemma \ref{lem:properties_of_G}\ref{itm:concavity} we have that $s\mapsto F(\alpha,s)$ is non decreasing. But then, we get from \eqref{eq:convextiy_propagation_2} and the last bound that 
\bes
\mathrm{Var}_{X\sim\hat{\pi}^y}(X_1)\geq\frac12G(\alpha,2), \quad \forall y\in\bbRD.
\ees 
Next, we observe that, because of the fact that if $\varphi(\cdot)$ satisfies \eqref{eq:convexity_propagation_3} then so does $\varphi(\mathrm{O}\cdot)$ for any orthonormal matrix $\mathrm{O}$, repeating the argument above yields 
\bes
\mathrm{Var}_{X\sim\hat{\pi}^y}(\langle v,X\rangle)\geq \frac12G(\alpha,2), \quad \forall y,v\in\bbRD\,\, \text{s.t.}\,\, | v|=1.
\ees
{ Recalling that $\nabla^2\bar\psi(y)=\frac1T \mathrm{Cov}_{X\sim\hat{\pi}^y}(X)$ with $\bar{\psi}$ defined by \eqref{eq:change_variable}, we find}
\be\label{eq:conv_prop_9}
\langle v,\nabla^2\bar\psi(y)v\rangle \geq \frac{G(\alpha,2)}{2T}|v|^2 \quad \forall v,y\in\bbRD.
\ee
But then, { rewriting \eqref{eq:change_variable} as}
\bes
\psi(\cdot) = U^{\nu}(\cdot)-\frac{|\cdot|^2}{2T}+\frac{\bar\psi(\cdot)}{T},
\ees
we immediately obtain that for all $r>0$
{
\bes
\begin{split}
\kappa_{\psi}(r) &\geq \kappa_{U^{\nu}}(r) -\frac{1}{T}+ \frac1T\kappa_{\bar\psi(y)}(r)\\
&\geq \alpha_{\nu}-\frac1T+\frac{G(\alpha,2)}{2T^2}-r^{-1}f_L(r),
\end{split}
\ees
where we used \eqref{eq:conv_prop_9} and hypothesis \eqref{eq:nu_ass} to obtain the last inequality.
}
\end{proof}

\subsection{Proof of Theorem \ref{thm:semiconvexity_estimate_Schr_pot}}
The proof is obtained by combining the results of the former two sections through a fixed point argument.
\begin{proof}[Proof of Theorem \ref{thm:semiconvexity_estimate_Schr_pot}]
We define a sequence $(\alpha^n)_{n\geq0}$ via 
\bes
\alpha^{0} =\alpha_{\nu}-\frac1T, \quad \alpha^{n}= \alpha_{\nu}-\frac1T+\frac{G(\alpha^{n-1},2)}{2T^2}, \quad n\geq 1.
\ees
{ Using Lemma \ref{lem:properties_of_G}\ref{itm:2} and an induction argument, we obtain that $\alpha^{1}\geq \alpha^0$ and $(\alpha^n)_{n\geq0}$ is a non decreasing sequence. Moreover, $(\alpha^n)_{n\geq0}$ is a bounded sequence by \eqref{eq:properties_of_G_1} and therefore it admits a finite limit $\alpha^*$}. By continuity of $G(\cdot,2)$, we know that $\alpha^*>\alpha_{\nu}-1/T$ and $\alpha^*$ satisfies the fixed point equation \eqref{eq:fixed_point}. To conclude the proof, we show by induction that 
\be\label{eq:semiconvexity_estimate_Schr_pot_1}
\kappa_{\psi}(r) \geq \alpha^n - r^{-1}f_{L}(r) \quad \forall n\geq 1.
\ee
The case $n=0$ is \eqref{eq:crude_bound}. For the inductive step, suppose \eqref{eq:semiconvexity_estimate_Schr_pot_1} holds for a given $n$. Then Lemma \ref{lem:concavity_propagation} gives that 
\bes
\ell_{\varphi}(r)\leq r^{-2}F(\alpha^n,r^2)-\frac1T \quad \forall r>0.
\ees
But then, an application of Lemma \ref{lem:convexity_propagation} proves  that for all $r>0$ we have
\bes
\kappa_{\psi}(r)\geq \alpha_{\nu}-\frac1T+\frac{G(\alpha^n,2)}{2T^2}-r^{-1}f_{L}(r)=\alpha^{n+1}-r^{-1}f_{L}(r).
\ees
The proof of \eqref{semiconvexity_estimate_Schr_pot} is now finished. To conclude, we observe that \eqref{semiconcavity_estimate_Schr_pot} follows directly from \eqref{semiconvexity_estimate_Schr_pot} and Lemma \ref{lem:convexity_propagation}.

\end{proof}
{ 
Let us now prove Corollary \ref{cor:rough}
\begin{proof}[Proof of Corollary \ref{cor:rough}]
    We first prove the upper bound. To do so, we observe that 
    \bes
    F(\alpha,s) \geq (\beta_{\mu} + \frac{1}{T(1+T\alpha)})s, \quad \forall \alpha\geq\alpha_{\nu}-1/T, s\geq 0.
    \ees
    But then, 
    \bes
        \frac{G(\alpha,2)}{2T^2} \leq \frac{1}{T^2\beta_{\mu} + \frac{1}{(\alpha+1/T)}}.
    \ees
    Since $\bar\alpha$ is a fixed point, we obtain
    \bes
        \bar\alpha+\frac{1}{T} \leq \alpha_{\nu}+\frac{1}{T^2\beta_{\mu} + \frac{1}{(\bar\alpha+1/T)}}.
    \ees
    If we now define $\bar{a}=\bar\alpha+1/T$, the above implies
    \bes
     \bar{a} \leq\alpha_{\nu}+ \frac{\bar{a}}{T^2\beta_{\mu}\bar{a} + 1}.
    \ees
   Since $\bar{a}>0$, we can rewrite the last inequality in the equivalent form
   \bes
    T^2\beta_{\mu}\bar{a}^2-T^2\alpha_{\nu}\beta_{\mu}\bar{a}-\alpha_{\nu} \leq 0.
   \ees
   Solving this differential inequality yields
   \bes
        \bar{a} \leq \frac{\alpha_{\nu}}{2} + \frac{1}{2}\sqrt{\alpha^2_{\nu}+\frac{4\alpha_{\nu}}{T^2\beta_{\mu}}}.
   \ees
   The desired result follows from $\bar\alpha=\bar{a}-\frac{1}{T}$. We now move to the proof of the lower bound. First, we recalling that $f_L(r)\leq Lr$, we obtain 
   \bes
        F(\alpha,s) \leq \big(\beta_{\mu} + \frac{1}{T(1+T\alpha)} + \frac{L}{(1+T\alpha)^2}\big)s.
   \ees
   Using that $\bar{\alpha}\geq \alpha_{\nu}-1/T$, we obtain that for all $s>0$
   \bes
        F(\bar\alpha,s) \leq \big(\beta_{\mu} + \frac{1}{T(1+T\bar\alpha)} + \frac{L}{T\alpha_{\nu}(1+T\bar\alpha)}\big)s.
   \ees
   But then,
   \bes
                \frac{G(\bar\alpha,2)}{2T^2} \geq \frac{1}{T^2\beta_{\mu} + \frac{(1+L/\alpha_{\nu})}{(\bar\alpha+1/T)}}.
   \ees
   Setting $\bar{a}=\alpha+1/T$, we deduce from the fact that $\bar{\alpha}$ is a fixed point that 
    \bes
        \bar{a} \geq \alpha_{\nu}+\frac{\bar{a}}{\bar{a}T^2\beta_{\mu} +(1+L/\alpha_{\nu}) }.
    \ees
    Using $\bar{a}>0$ we rewrite the last inequality in the equivalent form
    \bes
        T^2\beta_{\mu}\bar{a}^2+(L/\alpha_{\nu}-T^2\alpha_{\nu}\beta_{\mu})\bar{a}-(\alpha_{\nu}+L) \geq 0.
    \ees
       Solving this differential inequality yields
    \bes
        \bar{a} \geq   \frac{\alpha_{\nu}}{2} +\frac{L}{2T^2\alpha_{\nu}\beta_{\mu}}+\frac12\sqrt{\big(\alpha_{\nu} +\frac{L}{T^2\alpha_{\nu}\beta_{\mu}}\big)^2+\frac{4\alpha_{\nu}}{T^2\beta_{\mu}}}\,\,.  
    \ees
    The desired conclusion follows from $\bar{\alpha}= \bar{a}-1/T$.
\end{proof}
}

\section{Logarithmic Sobolev inequality for Schr\"odinger bridges}\label{sec:LSI}
This section is devoted to the proof of Theorem \ref{thm:LSI} and is structured as follows: we first recall known facts about logarithmic Sobolev inequalities and gradient estimates for diffusion semigroups whose proofs can be found e.g. in \cite{bakry2013analysis} and eventually prove at Lemma \ref{lem:mixing} a sufficient condition for the two-times distribution of a diffusion process to satisfy LSI. Though such a result may not appear surprising, we could not find it in this form in the existing literature.
We then proceed to elucidate the connection between Schr\"odinger bridges and Doob $h$-transforms at Lemma \ref{lem:doob}, and then finally prove Theorem \ref{thm:LSI}.
\paragraph{Local LSIs and gradient estimates} {Let $[0,T']\times \bbRD\ni(t,x)\mapsto U_t(x)$ be continuous in the time variable, twice differentiable and uniformly Lipschitz in the space variable.} We consider the time-inhomogeneous semigroup $(P_{s,t})_{0\leq s\leq t\leq T'}$ generated by the diffusion process whose generator at time $t$ acts on smooth functions with bounded support as follows
\bes 
f \mapsto \frac{1}{2}\Delta f -\langle \nabla U_t,\nabla f\rangle .
\ees 
{Moreover, we define for all $t\in[0,T']$
\bes
\alpha_t= \inf_{x,v\in\bbRD, |v|=1} \langle v,\nabla^2 U_t(x),v\rangle.
\ees}
We now recall some basic fact about gradient estimates and local LSIs for the semigroup $(P_{s,t})_{0\leq s\leq t\leq T'}$. For time-homogeneous semigroups these facts are well known and can be found e.g. in \cite{bakry2013analysis}: the adaptation to the time-inhomogeneous setting is straightforward. 
The first result we shall need afterwards is the gradient estimate (see \cite[Thm. 3.3.18]{bakry2013analysis})
\be\label{eq:grad_est}
|\nabla P_{t,T'}f|(x) \leq C_{t,T'} \,P_{t,T'}(|\nabla f|)(x), \quad C_{t,T'}=\exp\Big(-\int_t^{T'}\alpha_s\De s \Big),
\ee
that holds for all $(t,x)\in[0,T']\times\bbRD$ and any continuously differentiable $f$. Moreover, the local logarithmic Sobolev inequalities (see \cite[Thm. 5.5.2]{bakry2013analysis})
\be\label{eq:local_LSI}
( P_{0,T'}f\log f)(x)- (P_{0,T'}f )(x)\log( P_{0,T'}f)(x)\leq \frac{\tilde{C}_{0,T'}}{2} P_{0,T'}(|\nabla f|^2/f)(x), \quad \tilde{C}_{0,T'}=\int_{0}^{T'}C_{t,T'}\,\De t
\ee
hold  for all $x\in\bbRD$ and all positive continuously differentiable $f$. In the next Lemma we show how to obtain LSI for the joint law at times $0$ and $T'$ of a diffusion process with initial distribution $\mu$ and drift $-\nabla U_t$, that is to say for the coupling $\pi$ defined by 
\be\label{eq:diffusion_coupling}
\int_{\bbRD\times\bbRD} f(x,y)\pi(\De x\De y) = \int_{\bbRD} P_{0,T'}f(x,\cdot)(x)\mu(\De x) \quad \forall f>0.
\ee
\begin{lemma}\label{lem:mixing}
Assume that $\mu$ satisfies LSI with constant $C_\mu$ and let $\pi$ be as in \eqref{eq:diffusion_coupling}. Then $\pi$ satisfies LSI with constant
\bes 
\max\{2C_{\mu},{2}C_{\mu}C_{0,T'}+ \int_{0}^{T'}C_{t,T'} \De t\}.
\ees
\end{lemma}
{The} proof is carried out {by carefully "mixing" the local (conditional) LSIs \eqref{eq:local_LSI} with the help of gradient estimates}. Similar arguments and ideas can be found e.g. in \cite{bodineau1999log,grunewald2009two}.
\begin{proof}
Let $f>0$ be continuously differentiable. We recall the decomposition of the entropy formula (see \cite[Thm. 2.4]{leonard2014some})
\be\label{eq:ent_dec}
\ent_{\pi}(f) = \ent_{\mu}(f_0) + \int_{\bbRD} \ent_{\pi^x} (f^x) f_0(x) \mu(\De x),
\ee
where we adopted the following conventions
\be\label{eq:conventions}
f_0(x) = (P_{0,T'}f(x,\cdot))(x), \quad f^x(y)=f(x,y)/f_0(x), \quad \int g(y)\pi^x(\De y) = \Big(P_{0,T'}g\Big)(x) \,\,\forall g>0.\ee
{ The proof is carried out in two steps. In a first step, we bound the second term in \eqref{eq:ent_dec} by means of the conditional LSIs. In the second step, we bound the first term using the LSI for $\mu$ and gradient estimates.}
\begin{itemize}
\item\underline{Step 1} { The local logarithmic Sobolev inequalities \eqref{eq:local_LSI} imply that 

\bes
\begin{split}
\ent_{\pi^x} (f^x)& = P_{0,T'}\big( f^x \log f^x\big)(x)-\big(P_{0,T'}f^x \log P_{0,T'}f^x\big)(x) \\ &\leq \frac{\tilde{C}_{0,T'}}{2f_0(x)} \,\int |\nabla_y f^x(y)|^2/f^x(y) \pi^x(\De y) \\
& \stackrel{\eqref{eq:conventions}}{\leq}\frac{\tilde{C}_{0,T'}}{2f_0(x)} \,\int |\nabla_y f(x,y)|^2/f(x,y) \pi^x(\De y).
\end{split}
\ees
uniformly in $x\in\bbR^d$.Integrating this inequality and using \eqref{eq:diffusion_coupling} gives}
\be\label{eq:mixing_1}
\int \ent_{\pi^x} (f^x) f_0(x) \mu(\De x) \leq \frac{\tilde{C}_{0,T'}}{2}\,\int \frac{|\nabla_y f(x,y)|^2}{f(x,y)}\pi(\De x \De y).
\ee
\item \underline{Step 2}
{ We start with the observation that 
\bes
\nabla_x f_0(x) \stackrel{\eqref{eq:conventions}}{=}  P_{0,T'}  \nabla_xf(x,\cdot)(x) + \nabla_z\big(P_{0,T'}f(x,\cdot)(z) \big)\Big|_{z=x}.
\ees
}
But then, using the LSI for $\mu$ { and Young's inequality} we obtain
\be\label{eq:mixing_4}
\begin{split}
\ent_{\mu}(f_0) &\leq \frac{C_{\mu}}{2} \int |\nabla_x f_0(x)|^2/f_0(x)\mu(\De x)  \\
&\leq C_{\mu}\int |P_{0,T'}(\nabla_x f(x,\cdot))(x)|^2(P_{0,T'}f(x,\cdot))^{-1}(x)\,\mu(\De x)\\
&+ C_{\mu} \int |\nabla_z P_{0,T'}(f(x,\cdot))(z)|^2\Big|_{z=x} (P_{0,T'}f(x,\cdot))^{-1}(x)\,\mu(\De x).
\end{split}
\ee
For the first summand  on the rhs of \eqref{eq:mixing_4}, we can argue on the basis of Jensen's inequality applied to the convex function $a,b\mapsto a^2/b$ to obtain
\be\label{eq:mixing_2}
\begin{split}
&  \int |P_{0,T'}(\nabla_x f(x,\cdot))(x)|^2(P_{0,T'}f(x,\cdot))^{-1}(x)\mu(\De x)\\
&\leq  C_{\mu}\int P_{0,T'} \Big(|\nabla_x f(x,\cdot)|^2/f(x,\cdot)\Big)(x)\mu(\De x)\\
&= C_{\mu}\int |\nabla_x f(x,y)|^2/f(x,y) \pi(\De x \De y),
\end{split}
\ee
{ where we used \eqref{eq:diffusion_coupling} to obtain the last identity.} For the second summand on the rhs of \eqref{eq:mixing_4}, we first invoke the gradient estimate \eqref{eq:grad_est} and eventually apply again Jensen's inequality { as we did in the previous calculation} to obtain
\be\label{eq:mixing_3}
\begin{split}
&C_{\mu} \int |\nabla_z P_{0,T'}(f(x,\cdot))(z)|^2\Big|_{z=x} (P_{0,T'}f(x,\cdot))^{-1}(x)\mu(\De x) \\
&{\stackrel{\eqref{eq:grad_est}}{\leq} }C_{\mu}C_{0,T'} \int (P_{0,T'}(|\nabla_yf(x,\cdot)|)(x))^2 (P_{0,T}f(x,\cdot))^{-1}(x)\mu(\De x)\\
&{\stackrel{\text{Jensen}}{\leq} }C_{\mu}C_{0,T'}\int P_{0,T'}\Big(|\nabla_yf(x,\cdot)|^2/f(x,\cdot)\Big)(x)\mu(\De x)\\
&{\stackrel{\eqref{eq:diffusion_coupling}}{=}}C_{\mu} C_{0,T'}\int|\nabla_yf(x,y)|^2/f(x,y)\,\pi(\De x \De y).
\end{split}
\ee
{ Plugging in \eqref{eq:mixing_2}-\eqref{eq:mixing_3} into \eqref{eq:mixing_4} we obtain 
\be\label{eq:mixing_10}
\begin{split}
\ent_{\mu}(f_0) &\leq C_{\mu}\int |\nabla_x f(x,y)|^2/f(x,y) \pi(\De x \De y)\\
&+C_{\mu} C_{0,T'}\int|\nabla_yf(x,y)|^2/f(x,y)\,\pi(\De x \De y).
\end{split}
\ee}
\end{itemize}
{To conclude the proof, we combine the strength of the bounds \eqref{eq:mixing_1} and \eqref{eq:mixing_10} with the entropy decomposition formula \eqref{eq:ent_dec} to obtain
\begin{equation*}
\begin{split}
\ent_{\pi}(f)&\leq C_{\mu}\int |\nabla_x f(x,y)|^2/f(x,y) \pi(\De x \De y)\\
&+(\tilde{C}_{0,T'}/2+C_{\mu} C_{0,T'})\int|\nabla_yf(x,y)|^2/f(x,y)\,\pi(\De x \De y)\\
&\leq \max\{C_{\mu},C_{\mu}C_{0,T'}+ \tilde{C}_{0,T'}/2 \} \int|\nabla f|^2/f\,\De\pi,
\end{split}
\end{equation*}}
which is the desired result.
\end{proof}
In the next lemma, we represent an approximated version of the static Schr\"odinger bridge \eqref{eq:opt_coup} through a diffusion process. It is a classical result saying that Schr\"odinger bridges are indeed Doob's h-transforms, see e.g. \cite[Sec. 4]{LeoSch}\cite{dai1990markov}.
{
\begin{lemma}\label{lem:doob}
Let Assumption \ref{ass:marginals} hold and $\hat\pi$ be the static Schr\"odinger bridge \eqref{eq:opt_coup}. For any $\varepsilon\in(0,T)$ define 
\be\label{eq:opt_coup_approx}
\hat{\pi}^{\varepsilon}(\De x\,\De y) :=(2\pi (T-\varepsilon))^{-d/2}\exp\Big(-\varphi(x)-\hjb{T}{\psi}{T-\varepsilon}(y)-\frac{|y-x|^2}{2(T-\varepsilon)}\Big)\De x \,\De y.
\ee
Then $\hat{\pi}^{\varepsilon}$ has the form \eqref{eq:diffusion_coupling} for $T'=T-\varepsilon$, where $(P_{s,t})_{0\leq s\leq t\leq T-\varepsilon}$ is the time-inhomogeneous semigroup associated with the generator acting on smooth test functions as follows
\be\label{eq:doob_1} 
f \mapsto \frac{1}{2}\Delta f-\langle\nabla \hjb{T}{\psi}{t},\nabla f\rangle, \quad  \,\, t\in[0,T-\varepsilon].\ee
\end{lemma}

\begin{proof}
Let $\psi$ be the Schr\"odinger potential in \eqref{eq:schr_syst}. 
Invoking Theorem \ref{thm:semiconvexity_estimate_Schr_pot} we obtain that \eqref{eq:concavity_propagation_1} holds with $\alpha=\alpha_{\psi}$. Moreover, it is well known that (see \cite[Eq 3.3]{mikulincer2021brownian} for example)
$\ell_{\hjb{T}{\psi}{t}} \leq (T-t)^{-1}$, i.e. the Hessian of $\hjb{T}{\psi}{t}$ is bounded above by $(T-t)^{-1}$. Thereofore, the vector field $[0,T-\varepsilon]\times\bbR^d \ni(t,x)\mapsto-\nabla\hjb{T}{\psi}{t}(x) $ is uniformly Lipschitz w.r.t. the space variable for any $\varepsilon\in(0,T)$. This classically implies existence and uniqueness of strong solutions for the stochastic differential equation 
\be\label{eq:char}
\De X_t = -\nabla \hjb{T}{\psi}{t}(X_t)\De t +\De B_t, \quad X_0\sim\mu
\ee
over any time interval $[0,T-\varepsilon]$ and we shall denote by $\bbQ^{\varepsilon}$ the law of the solution on $C([0,T-\varepsilon];\bbRD)$.
Next, we denote by $\bbP^{\varepsilon}$ the law on law on $C([0,T-\varepsilon];\bbRD)$ of a Brownian motion started at $\mu$. By Girsanov's Theorem, see \cite{LeoGir} for a version that applies in the current setting, we know that 
\bes
\frac{\De\bbQ^{\varepsilon}}{\De\bbP^{\varepsilon}}(\omega) = \exp\Big(-\int_0^{T-\varepsilon} \nabla\hjb{T}{\psi}{t}(\omega_t)\De \omega_t -\frac{1}{2}\int_0^{T-\varepsilon} |\nabla\hjb{T}{\psi}{t}(\omega_t)|^2\De t \Big) \quad \bbP^{\varepsilon}-\text{a.s.},
\ees
where we denote by $\omega$ the typical element of the canonical space $C([0,T-\varepsilon];\bbRD)$.
Using It\^o formula we rewrite the above as
\bes
\begin{split}
\frac{\De\bbQ^{\varepsilon}}{\De\bbP^{\varepsilon}}(\omega)& = \exp\Big(\hjb{T}{\psi}{0}(\omega_0)-\hjb{T}{\psi}{T-\varepsilon}(\omega_{T-\varepsilon})+\int_0^{T-\varepsilon}\Big(\partial_t \hjb{T}{\psi}{t}+\frac12\Delta\hjb{T}{\psi}{t}-\frac{1}{2}|\nabla\hjb{T}{\psi}{t}|^2\Big)(\omega_t)\De t \Big)\\
&=\exp(U^{\mu}(\omega_0)-\varphi(\omega_0)-\hjb{T}{\psi}{T-\varepsilon}(\omega_{T-\varepsilon}))
\end{split}
\ees
where we used the  Schr\"odinger system \eqref{eq:schr_syst} and the HJB equation \eqref{eq:HJB} to obtain the last expression. Indeed because of Theorem \ref{thm:semiconvexity_estimate_Schr_pot} one can deduce that $[0,T]\times\bbRD\ni(t,x)\mapsto\hjb{T}{\psi}{t}(x)$, is a classical solution of \eqref{eq:HJB} by differentiating under the integral sign in \eqref{eq:HJB_FK}. From this, we deduce that
\bes
\frac{\De\bbQ^{\varepsilon}_{0,T-\varepsilon}}{\De\bbP^{\varepsilon}_{0,T-\varepsilon}}(x,y)=\exp\big(U^{\mu}(x)-\varphi(x)-\hjb{T}{\psi}{T-\varepsilon}(y)\big) \quad \bbP^{\varepsilon}_{0T}-\text{a.s.},
\ees
where $\bbQ^{\varepsilon}_{0,T-\varepsilon}$ (resp. $\bbP^{\varepsilon}_{0,T-\varepsilon}$) denotes the joint distribution of $\bbQ^{\varepsilon}$ (resp. $\bbP^{\varepsilon}$) at times $0$ and $T-\varepsilon$. Since 
\bes
\bbP^{\varepsilon}_{0,T-\varepsilon}(\De x \,\De y)= (2\pi (T-\varepsilon))^{-d/2}\exp(-U^{\mu}(x))\exp\Big(-\frac{|y-x|^2}{2(T-\varepsilon)}\Big)\De x \,\De y,
\ees 
we conclude that 
\bes
\bbQ^{\varepsilon}_{0,T-\varepsilon}(\De x \De y ) =(2\pi (T-\varepsilon))^{-d/2}\exp\Big(-\varphi(x)-\hjb{T}{\psi}{T-\varepsilon}(y)-\frac{|y-x|^2}{2(T-\varepsilon)}\Big)\De x \,\De y.
\ees
But then $\bbQ^{\varepsilon}_{0,T-\varepsilon}=\hat{\pi}^{\varepsilon}$, where $\hat{\pi}^{\varepsilon}$ is defined at \eqref{eq:opt_coup_approx}. To conclude, we recall that $\bbQ_{0,T-\varepsilon}$ has the desired form \eqref{eq:diffusion_coupling} where $(P_{s,t})_{0\leq s\leq t\leq T-\varepsilon}$ is indeed the semigroup generated by \eqref{eq:doob_1}.
\end{proof}}

\begin{proof}[Proof of Theorem \ref{thm:LSI}]
We know by Lemma \ref{lem:doob} that $\hat{\pi}^{\varepsilon}$ has the form \eqref{eq:diffusion_coupling} for $T'=T-\varepsilon $ and the inhomogeneous semigroup generated by \eqref{eq:doob_1}. We now set for $t\in[0,T]$
\bes 
\alpha^{\psi}_t= \inf_{x,v\in\bbRD, |v|=1} \langle v,\nabla^2 \hjb{\psi}{T}{t}(x),v\rangle 
\ees 
and proceed to estimate $\alpha^{\psi}_t$ from below. { Invoking Theorem \ref{thm:semiconvexity_estimate_Schr_pot} we obtain that \eqref{eq:concavity_propagation_1} holds with $\alpha=\alpha_{\psi}$. That is to say, the estimate} 
\bes
\kappa_{U^{T,\psi}_t}(r)  \geq \frac{\alpha^{\psi}}{1+(T-t)\alpha^{\psi}}-  \frac{r^{-1}f_L(r)}{(1+(T-t)\alpha^{\psi})^2} 
\ees
{ holds uniformly on $r>0$ and $0\leq t\leq T$.}
From here, using the concavity of $f_L$ and $f'_L(0)=L$ we obtain
\bes
\alpha^{\psi}_t \geq \frac{\alpha^{\psi}}{1+(T-t)\alpha^{\psi}}-  \frac{L}{(1+(T-t)\alpha^{\psi})^2}.
\ees
{ We can now apply Lemma \ref{lem:mixing} to obtain that $\hat{\pi}^{\varepsilon}$ satisfies LSI with constant given by
\bes
\eta_{\varepsilon}:=\max\{2C_{\mu},{2}C_{\mu}C_{0,T-\varepsilon}+ \int_{0}^{T-\varepsilon}C_{t,T-\varepsilon}\De t \}.
\ees
Next, observe that the weak convexity bounds on $\psi$ of Theorem \ref{thm:semiconvexity_estimate_Schr_pot} imply that $\hjb{T}{\psi}{T-\varepsilon}$ converges to $\psi$ pointwise as $\varepsilon\rightarrow0$. But then, we have that $\hat{\pi}^{\varepsilon}$ converges in total variation to $\hat{\pi}$ by Scheff\'e's Lemma. Take now any continuously differentiable function $f$ bounded above and below by positive constants and with bounded derivative. Letting $\varepsilon\rightarrow0$ in
\bes
\mathrm{Ent}_{\hat{\pi}^{\varepsilon}}(f) \leq \frac{\eta_{\varepsilon}}{2} \int \frac{|\nabla f|^2}{f}(x,y)\,\hat{\pi}^{\varepsilon}(\De x\, \De y)
\ees
 and using the convergence in variation of $\hat{\pi}^{\varepsilon}$ obtain that LSI holds for $f$ under $\hat{\pi}$ with the desired constant. The extension to a general positive and continuously differentiable functions is achieved through a standard approximation argument where $f(\cdot)$ is approximated by $f\,\chi(N^{-1}\cdot)+N^{-1}$ with $\chi(\cdot)$ a smooth cutoff function. 
}

\end{proof}

\section{Appendix}
\begin{prop}\label{prop:gen_ass}
Assume that $U$ satisfies \eqref{eq:Eberle_ass} for some $\alpha>0,L',R\geq0$. Then 
\bes
\kappa_U(r) \geq \alpha-r^{-1}f_L(r) \quad \forall r>0.
\ees
with $L$ given by \eqref{eq:bar_L_L}.
\end{prop}
\begin{proof}
If $r> R$ the claim is a simple consequence of $f_{L}(r)\geq0$. If $r\leq R$, using \eqref{eq:basic properties} to get that $r'\mapsto r'^{-1}f_L(r')$ is non increasing on $(0,+\infty)$, we obtain 
\bes
 r^{-1} f_{L}(r) \geq R^{-1}f_{L}(R)=L',
\ees
from which the conclusion follows.
\end{proof}
\begin{prop}\label{prop:hessian_cov}
Let Assumption \ref{ass:marginals} hold and assume furthermore that there exist $\varepsilon,\gamma'>0$ such that 
\be\label{eq:integr_ass}
\int \exp(\gamma'|x|^{1+\varepsilon})\mu(\De x) <+\infty.
\ee
Moreover, let $\bar\psi$ be as in \eqref{eq:change_variable}. Then $\bar\psi$ is twice differentiable and we have
\bes
\nabla^2\bar\psi(y)=\frac1T \mathrm{Cov}_{X\sim\hat{\pi}^y}(X) \quad \forall y\in \bbRD, 
\ees
where $\hat{\pi}^y$ is given by \eqref{eq:cond_distr}.
\end{prop}

\begin{proof}
From \eqref{eq:schr_syst} we obtain that
\be\label{eq:der_cov_2}
\bar{\psi}(y)+ \frac{d}{2}\log(\pi)=T \log \int_{\bbRD} \exp\Big(-\varphi(x)-\frac{|x|^2}{2T}+\frac{\langle x,y\rangle}{T}\Big)\De x. 
\ee
From Assumption \ref{ass:marginals}, \eqref{eq:schr_syst} and \eqref{eq:integr_ass} it follows that 
\bes
\int_{\bbRD\times\bbRD} \exp\Big(\gamma'|x|^{1+\varepsilon} -\varphi(x)-\psi(y)-\frac{|x-y|^2}{2T}\Big)\De x\,\De y<+\infty,
\ees
whence the existence of some $y'$ such that 
\bes
\int_{\bbRD\times\bbRD} \exp\left(\gamma'|x|^{1+\varepsilon} -\varphi(x)-\frac{|x|^2}{2T} + \frac{\langle x,y'\rangle}{T}\right)\De x <+\infty.
\ees
From this, we easily obtain that for all $\gamma<\gamma'$
\be\label{eq:der_cov_1}
\int_{\bbRD} \exp\left(\gamma|x|^{1+\varepsilon} -\varphi(x)-\frac{|x|^2}{2T} + \frac{\langle x,y\rangle}{T}\right)\De x <+\infty \quad \forall y\in\bbRD.
\ee
Thanks to \eqref{eq:der_cov_1} we can apply the dominated convergence theorem and differentiate under the integral sign in \eqref{eq:HJB} to obtain that $\bar\psi$ is differentiable and 
\bes
\nabla\bar\psi(y) =\frac{\int x \exp(-\varphi(x)-\frac{|x|^2}{2T}+\frac{\langle x,y\rangle}{T})\De x}{\int \exp(-\varphi(\bar{x})-\frac{|\bar{x}|^2}{2T}+\frac{\langle \bar{x},y\rangle}{T})\De \bar{x}} \stackrel{\eqref{eq:cond_distr}}{=} \bbE_{X\sim \hat{\pi}^y}[X]
\ees
Using once again \eqref{eq:der_cov_1} to differentiate under the integral sign in \eqref{eq:der_cov_2} we conclude that $\bar\psi$ is twice differentiable and that \eqref{eq:convexity_propagation_4} holds.
\end{proof}

\bibliographystyle{plain}
\bibliography{Refbis}
\end{document}